\documentclass{mcom-l}
\usepackage{amssymb}

\newtheorem{thm}{Theorem}[section]
\newtheorem{cor}[thm]{Corollary}
\newtheorem{lem}[thm]{Lemma}
\newtheorem{prop}[thm]{Proposition}
\newtheorem{con}[thm]{Conjecture}
\theoremstyle{definition}

\numberwithin{equation}{section}
\newcommand{\norm}[1]{\left\Vert#1\right\Vert}

\newcommand{\eps}{\varepsilon}

\newcommand{\ve}{\varepsilon}
\newcommand{\Q}{\mathbb{Q}}
\newcommand{\R}{\mathbb{R}}
\newcommand{\N}{\mathbb{N}}
\newcommand{\Z}{\mathbb{Z}}
\newcommand{\C}{\mathbb{C}}

\newcommand{\M}{\mathcal{M}}

\newcommand{\Notequiv}{/\kern-.6em\hbox{$\equiv$} }

\begin{document}

\title[Monic integer Chebyshev problem]
{Monic integer Chebyshev problem}%

\author{P. B. Borwein}%
\address{Department of Mathematics and Statistics, Simon Fraser
University, Burnaby, B. C., V5A 1S6, Canada}%
\email{pborwein@cecm.sfu.ca}

\author{C. G. Pinner}%
\address{Department of Mathematics, 138 Cardwell Hall,
Kansas State University, Manhattan, KS 66506, U.S.A.}
\email{pinner@math.ksu.edu}

\author{I. E. Pritsker}
\address{Department of Mathematics, 401 Mathematical Sciences,
Oklahoma State University, Stillwater, OK 74078, U.S.A.}%
\email{igor@math.okstate.edu}

\thanks{Research of the authors was supported in part by the
following grants: NSERC of Canada and MITACS (Borwein), NSF grant
EPS-9874732 and matching support from the state of Kansas
(Pinner), and NSF grant DMS-9996410 (Pritsker).}%



\begin{abstract}

We study the problem of minimizing the supremum norm by monic
polynomials with integer coefficients. Let ${\M}_n({\Z})$ denote
the monic polynomials of degree $n$ with integer coefficients. A
{\it monic integer Chebyshev polynomial} $M_n \in {\M}_n({\Z})$
satisfies
$$
\| M_n \|_{E} = \inf_{P_n \in{\M}_n ( {\Z})} \| P_n \|_{E}.
$$
and the {\it monic integer Chebyshev constant} is then defined by
$$
t_M(E) := \lim_{n \rightarrow \infty} \| M_n \|_{E}^{1/n}.
$$
This is the obvious analogue of the more usual {\it integer
Chebyshev constant} that has been much studied.

We compute $t_M(E)$ for various sets including all finite sets of
rationals and make the following conjecture, which we prove in
many cases.

\medskip\noindent {\bf Conjecture.} {\it Suppose
$[{a_2}/{b_2},{a_1}/{b_1}]$ is an interval whose endpoints are
consecutive Farey fractions. This is characterized by
$a_1b_2-a_2b_1=1.$ Then}
$$t_M[{a_2}/{b_2},{a_1}/{b_1}] = \max(1/b_1,1/b_2).$$

This should be contrasted with  the non-monic integer Chebyshev
constant case where the only intervals where the constant is
exactly computed are intervals of length 4 or greater.

\end{abstract}

\maketitle


\section{Introduction and general results}

Define the uniform (sup) norm on a compact set  $E \subset {\C}$
by
$$\norm{f}_{E} := \sup_{z \in E} |f(z)|.$$
We study the monic polynomials with integer coefficients that
minimize the sup norm on the set $E$. Let ${\mathcal P}_n ({\C})$
and ${\mathcal P}_n ({\Z})$ be the classes of algebraic
polynomials of degree at most $n$, respectively with complex and
with integer coefficients. Similarly, we define the classes of
monic polynomials ${\M}_n ({\C})$ and ${\M}_n ({\Z})$ of {\em
exact} degree $n\in\N$. The problem of minimizing the uniform norm
on $E$ by polynomials from ${\mathcal M}_n ( {\C})$ is well known
as the Chebyshev problem (see \cite{BE95}, \cite{Ri90},
\cite{Ts75}, \cite{Go69}, etc.) In the classical case, $E=[-1,1]$,
the explicit solution of this problem is given by the monic
Chebyshev polynomial of degree $n$:
$$T_n (x) := 2^{1 -n} \cos (n \arccos x), \quad n \in \N.$$
Using a change of variable, we can immediately extend this to an
arbitrary interval $[a,b] \subset {\R}$, so that
$$t_n (x) := \left( \frac{b -a}{2} \right)^n T_n \left(
\frac{2x -a -b}{b -a} \right)$$ is a monic polynomial with real
coefficients and the smallest uniform norm on $[a,b]$ among all
polynomials from ${\M}_n({\C} )$. In fact,
\begin{equation} \label{1.1}
\| t_n \|_{[a,b]} = 2 \left( \frac{b -a}{4} \right)^n, \quad n \in
{\N},
\end{equation}
and we find that the {\it Chebyshev constant} for $[a,b]$ is given
by
\begin{equation} \label{1.2}
t_{\C}([a,b]) := \lim_{n \rightarrow \infty} \| t_n
\|_{[a,b]}^{1/n} = \frac{b -a}{4}.
\end{equation}
The Chebyshev constant of an arbitrary compact set $E \subset
{\C}$ is defined in a similar fashion:
\begin{equation} \label{1.3}
t_{\C}(E) := \lim_{n \rightarrow \infty} \| t_n \|_{E}^{1/n},
\end{equation}
where $t_n$ is the Chebyshev polynomial of degree $n$ on $E$ (the
monic polynomial of exact degree $n$ of minimal supremum norm on
$E$). It is known that $t_{\C}(E)$ is equal to the transfinite
diameter and the logarithmic capacity $\textup{cap}(E)$ of the set
$E$ (cf. \cite[pp. 71-75]{Ts75}, \cite{Go69} and \cite{Ra95} for
the definitions and background material).

An {\it integer Chebyshev polynomial} $Q_n \in {\mathcal P}_n (
{\Z})$ for a compact set $E \subset \C$ is defined by
\begin{equation} \label{1.4}
\| Q_n \|_{E} = \inf_{0 \not\equiv P_n \in{\mathcal P}_n ( {\Z})}
\| P_n \|_{E},
\end{equation}
where the $\inf$ is taken over all polynomials from ${\mathcal
P}_n ( {\Z})$, which are not identically zero. Further, the {\it
integer Chebyshev constant} (or integer transfinite diameter) for
$E$ is given by
\begin{equation} \label{1.5}
t_{\Z}(E) := \lim_{n \rightarrow \infty} \| Q_n \|_{E}^{1/n}.
\end{equation}
The  integer Chebyshev problem is also a classical subject of
analysis and number theory (see \cite[Ch. 10]{Mo94}, \cite{Ch83},
\cite{BE96}, \cite{Fer80}, \cite{FRS97}, \cite{HS97}, \cite{Tr71},
\cite{Pr00} and the references therein). It does not require the
polynomials to be monic. We define the associated quantities for
the {\em monic} integer Chebyshev problem as follows. A {\it monic
integer Chebyshev polynomial} $M_n \in {\M}_n({\Z}),\ \deg M_n=n,$
satisfies
\begin{equation} \label{1.6}
\| M_n \|_{E} = \inf_{P_n \in{\M}_n ( {\Z})} \| P_n \|_{E}.
\end{equation}
The {\it monic integer Chebyshev constant} is then defined by
\begin{equation} \label{1.7}
t_M(E) := \lim_{n \rightarrow \infty} \| M_n \|_{E}^{1/n} = \inf_n
\| M_n \|_{E}^{1/n},
\end{equation}
where the existence of this limit and the last equality follows by
a standard argument presented in Lemma \ref{new}. The monic
integer Chebyshev problem is quite different from the classical
integer Chebyshev problem, as we show in this paper.

It is immediately clear from the definitions
(\ref{1.4})-(\ref{1.7}) that
\begin{equation} \label{1.8}
t_M(E) \ge t_{\Z}(E).
\end{equation}
Note that, for any $P_n \in{\mathcal P}_n ( {\Z})$,
$$ \| P_n \|_{E} = \| P_n \|_{E^*},$$
where $E^*:=E \cup \{z:{\bar z} \in E\}$, because $P_n$ has real
coefficients. Thus the (monic) integer Chebyshev problem on a
compact set $E$ is equivalent to that on $E^*$, and we can assume
that $E$ is symmetric with respect to the real axis
($\R$-symmetric) without loss of generality.

Our first result shows that the monic integer Chebyshev constant
coincides with the regular Chebyshev constant (capacity) for
sufficiently large sets.

\begin{thm} \label{thm1.1}
If $E$ is $\R$-symmetric and $\textup{cap}(E)\ge 1$, then
\begin{equation} \label{1.9}
t_M(E)=\textup{cap}(E).
\end{equation}
\end{thm}
We remark that $t_{\Z}(E) = 1$ for the sets $E$ with
$\textup{cap}(E) \ge 1.$ Indeed, $\norm{P_n}_E \ge
(\textup{cap}(E))^n$ for any $P_n \in {\mathcal P}_n ( {\Z})$ of
exact degree $n$ (cf. \cite[p. 155]{Ra95}). Thus $Q_n(z) \equiv 1$
is a minimizer for (\ref{1.4}) in this case.

An argument going back to Kakeya (cf. \cite{Ok24} or \cite{Tr71})
gives

\begin{prop} \label{prop1.2}
Let $E\subset\C$ be a compact $\R$-symmetric set. If
$\textup{cap}(E)<1$ then $t_M(E)<1.$
\end{prop}
We show below that this statement cannot be significantly
improved.

The monic integer Chebyshev constant shares a number of standard
properties with $t_{\Z}(E)$ and $t_{\C}(E)$, such as the
monotonicity property below.

\begin{prop} \label{prop1.3}
Let $E \subset F \subset \C$. Then
$$t_M(E) \le t_M(F).$$
\end{prop}

Another generic property of importance is the following (see
\cite{Fe30} and Theorem 2 of \cite[Sect. VII.1]{Go69}).

\begin{prop} \label{prop1.4}
Let $E \subset \C$ be a compact set. If $P_n^{-1}(E)$ is the
inverse image of $E$ under $P_n\in{\M}_n({\Z}),\ \deg P_n=n,$ then
\begin{equation} \label{1.10}
t_M\left(P_n^{-1}(E)\right)=\left(t_M(E)\right)^{1/n}.
\end{equation}
\end{prop}

Perhaps, the most distinctive feature of $t_M(E)$ is that it may
be different from zero even for a single point. For example (see
section 2 below), suppose that $m,n\in\Z,$ where $n\ge 2$ and
$(m,n)=1.$ Then
\begin{equation} \label{1.11}
t_M\left(\left\{\frac{m}{n}\right\}\right)=\frac{1}{n}.
\end{equation}

On the other hand, if $a\in\R$ is irrational, then
\begin{equation} \label{1.11a}
t_M\left(\left\{a\right\}\right)=0.
\end{equation}
This result has several interesting consequences. Consider
$E_n:=\{z:z^n=1/2\},\ n\in\N.$ It is obvious that
$\textup{cap}(E_n)= t_{\C}(E_n)=0$ for any $n\in\N.$ However,
(\ref{1.11}) and (\ref{1.10}) imply that $t_M(E_n)=2^{-1/n} \to
1,$ as $n\to\infty.$ Thus no uniform upper estimate of $t_M(E)$ in
terms of $\textup{cap}(E)$ is possible, in contrast with the
inequality $t_{\Z}(E)\le\sqrt{\textup{cap}(E)}$ (see the results
of Hilbert \cite{Hi1894} and Fekete \cite{Fe23}).

We also note that $t_M(\{1/\sqrt{2}\})=t_M(\{-1/\sqrt{2}\})=0$ by
(\ref{1.11a}), while $t_M(\{1/\sqrt{2}\}\cup\{-1/\sqrt{2}\})=
1/\sqrt{2}$ by (\ref{1.10}) and (\ref{1.11}). This shows that
another well known property of capacity is not valid for $t_M(E).$
Namely, capacity (Chebyshev constant) for the union of two sets of
zero capacity is still zero (cf. Theorem III.8 of \cite[p.
57]{Ts75}).

Combining Proposition \ref{prop1.3}, Proposition \ref{prop1.4} and
(\ref{1.11}), one can find the explicit values of the monic
integer Chebyshev constant for many intervals and other sets.

\begin{thm} \label{thm1.6}
Let $n\in\Z.$ The following relations hold true:
\begin{equation} \label{1.12}
t_M\left(\left[0,\frac{1}{n}\right]\right) =
t_M\left(\left[\frac{n-1}{n},1\right]\right) =
t_M\left(\left[-\frac{1}{n},\frac{1}{n}\right]\right) =
\frac{1}{n}, \quad n\ge 2,
\end{equation}

\begin{equation} \label{1.13}
t_M\left(\left[-\frac{1}{\sqrt{n}},\frac{1}{\sqrt{n}}\right]\right)
= \sqrt{t_M\left(\left[0,\frac{1}{n}\right]\right)} =
\frac{1}{\sqrt{n}}, \quad n\ge 2,
\end{equation}

\begin{equation} \label{1.14}
t_M\left(\left[n,n+\frac{1}{2}\right]\right) =
t_M\left(\left[n-\frac{1}{2},n\right]\right) =
t_M\left(\left[0,\frac{1}{2}\right]\right) = \frac{1}{2},
\end{equation}

\begin{equation} \label{1.15}
t_M\left(\left[n,n+1\right]\right) =
t_M\left(\left[0,1\right]\right) =
\sqrt{t_M\left(\left[0,\frac{1}{4}\right]\right)} = \frac{1}{2}
\end{equation}
and
\begin{equation} \label{1.16}
t_M\left(\left[n,n+2\right]\right) =
t_M\left(\left[-1,1\right]\right) =
\sqrt{t_M\left(\left[0,1\right]\right)} = \frac{1}{\sqrt{2}}.
\end{equation}

Also, if $E\subset[(1-\sqrt{2})/2,(1+\sqrt{2})/2]$ and $\{1/2\}
\in E$, then
\begin{equation} \label{1.17}
t_M\left(E\right) =
t_M\left(\left[\frac{1-\sqrt{2}}{2},\frac{1+\sqrt{2}}{2}\right]\right)
= t_M\left(\left\{\frac{1}{2}\right\}\right)= \frac{1}{2}.
\end{equation}
\end{thm}
Of course, the above list of values can be extended further. It is
worth mentioning that finding the value of $t_{\Z}([0,1])$ is a
notoriously difficult problem, where we do not even have a current
conjecture (see  \cite{Ch83}, \cite[Ch. 10]{Mo94}, \cite{BE96} and
\cite{Pr00}). From this point of view, the monic integer Chebyshev
problem seems to be easier than its classical counterpart.

The rest of our paper is organized as follows. We consider the
monic integer Chebyshev problem for finite sets in Section 2.
Sections 3 and 4 contain proofs of the results from Sections 1 and
2 respectively. Section 5 is devoted to the study of Farey
intervals, where we give some numerical results and state an
interesting conjecture on the value of the monic integer Chebyshev
constant.

\section{Finite Sets of Points}

While finite numbers of integers can of course in no way affect
$t_{M}(E)$, it is readily seen that the presence of  non-integer
rationals does restrict how small $t_M(E)$ can become, with
$$ \frac{a}{b}\in E, \;\;\;b\geq 2 \Rightarrow t_{M}(E)\geq \frac{1}{b}, $$
(for a monic integer polynomial $P$ of degree $n$ we plainly have
$|b^nP(a/b)|\geq 1$). Indeed for a finite set of rationals this
bound is precise, as an immediate  consequence of the following:

\begin{thm} \label{thm2.1}
For any $k$ rational points
$$ \frac{a_i}{b_i},\;\;\; (a_i,b_i)=1,\;\; i=1,\ldots,k, $$
there is a monic integer polynomial $f(x)$ of degree $n$ for some
positive integer $n$ with
$$ f\left(\frac{a_i}{b_i}\right)=\frac{1}{b_i^n},\;\;\; i=1,\ldots,k. $$
\end{thm}

\begin{cor} \label{cor2.1}
If $E=\left\{\frac{a_1}{b_1},\ldots,\frac{a_{k}}{b_{k}}\right\}$
with the $a_i/b_i$ rationals written in their lowest terms and
$b_i\geq 2$, then
$$ t_M(E)=\max_{i=1,\ldots,k} \frac{1}{b_i}. $$
\end{cor}

\noindent {\bf Two consecutive Farey fractions: } It is perhaps
worth noting that in the case of two   consecutive Farey fractions
$$ \frac{a_2}{b_2}<\frac{a_1}{b_1}, \;\;\; a_1b_2-a_2b_1=1, $$
it is easy to explicitly write down a polynomial satisfying
Theorem 2.1 (or any congruence feasible values):

If $n\geq 2$ with
$$ a_i^n\equiv A_i \mod b_i,\;\;\; i=1,2, $$
then
$$ f(x)=x^n+ \left(\frac{A_1-a_1^n}{b_1}\right) \left(b_2x-a_2\right)^{n-1}
+\left(\frac{A_2-a_2^n}{b_2}\right) \left(a_1-b_1x\right)^{n-1} $$
has $f(a_i/b_i)=A_i/b_i^n$, $i=1,2$.

Moreover if $a_3/b_3$ is the next Farey fraction between them,
$a_3=a_1+a_2$, $b_3=b_1+b_2$, and $n\geq 3$ with $a_i^n\equiv A_i
\mod b_i$, $i=1,2,3$, the polynomial
$$ \tilde{f}(x)=f(x) + \left( \frac{A_3-a_3^n}{b_3}- \frac{A_2-a_2^n}{b_2}-
\frac{A_1-a_1^n}{b_1}\right) (b_2x-a_2)^j(a_1-b_1 x)^{n-1-j}, $$
satisfies $\tilde{f}(a_i/b_i)=A_i/b_i^n$, $i=1,2,3$, for any
$1\leq j\leq n-2.$

For higher degree algebraic numbers which are not algebraic
integers (adding an algebraic integer can plainly not change the
monic integer Chebyshev constant), the presence of a full set of
conjugates similarly leads to a lower bound. In particular if $E$
contains all the roots $\alpha_1,\ldots,\alpha_d$ of an
irreducible integer polynomial of degree $d$ and lead coefficient
$b\geq 2$ then
$$ t_{M}(E)\geq \frac{1}{b^{\frac{1}{d}}}, $$
(since for any monic integer polynomial $P$ of degree $n$ the
quantity $b^n\prod_{i=1}^d P(\alpha_i)$ is an integer and
necessarily non-zero). Proposition \ref{prop1.4} and Corollary
\ref{cor2.1} can be used to furnish non-rational cases where such
a bound is sharp. However, if $E$ consists of a set of conjugates
missing at least one real or pair of complex conjugates, then in
fact  $t_M(E)=0$. Similarly if $E$ consists of a finite number of
transcendentals. These (and other similar examples) follow at once
from the following result:

\begin{thm} \label{thm2.2}
Suppose that $S=\{\alpha_1,\ldots,\alpha_k\}$ is a set of $k$
numbers, with the $\alpha_i$ transcendental or algebraic of degree
more than $k$. If $S$ is closed under complex conjugation, then
for any $\ve \in (0,1)$ there is a monic integer polynomial $F$ of
degree $n=n(\alpha_1,\ldots,\alpha_k)$ with $|F(\alpha_i)|<\ve$,
$i=1,\ldots,k$.
\end{thm}

\section{Proofs for section 1}

\begin{lem} \label{new}
The limit defining $t_M(E)$ in \eqref{1.7} by
\[t_M(E) := \lim_{n \rightarrow \infty} \| M_n \|_{E}^{1/n} \]
exists. Furthermore,
\[ \lim_{n \rightarrow \infty} \| M_n \|_{E}^{1/n} =
\inf_n \| M_n \|_{E}^{1/n}.\]
\end{lem}
\begin{proof}
The argument is identical to the classical Chebyshev constant
case. Indeed, let
\[ v_n :=\norm{M_n}_E = \inf_{P_n \in{\mathcal M}_n ( {\Z})}
\| P_n \|_{E},\quad n \in\N. \] Then
\[ v_{k+m} \le \norm{ M_k M_m}_E \le \norm{M_k}_E
\norm{M_m}_E = v_{k} v_{m}.\] On setting $a_n = \log{v_n}$, we
obtain that
\[a_{k+m} \le a_k + a_m,\quad k,m \in\N.\]
Hence
\[\lim_{n\to\infty} \frac{a_n}{n} = \lim_{n\to\infty} \log \left(
v_n \right)^{1/n}\] exists (possibly as $-\infty$) by Lemma on
page 73 of \cite{Ts75}.

If $\inf_n \| M_n \|_{E}^{1/n}=\liminf_{n\to\infty} \| M_n
\|_{E}^{1/n}$ then the second statement of this lemma follows from
the above. Otherwise, we have
\[\inf_n \| M_n \|_{E}^{1/n} = \| M_k \|_{E}^{1/k} \]
for a particular $k\in\N.$ But then the sequence of polynomials
$\{(M_k)^m\}_{m=0}^{\infty}$ satisfies
\[ \| M_k \|_{E}^{1/k} = \lim_{m \rightarrow \infty}
\| (M_k)^m \|_{E}^{1/(km)} \ge t_M(E) \ge \inf_n \| M_n
\|_{E}^{1/n} = \| M_k \|_{E}^{1/k}.\]
\end{proof}

\begin{proof}[Proof of Theorem \ref{thm1.1}]
Let $T_n(z)=z^n+a_{n-1}^{(n)} z^{n-1}+a_{n-2}^{(n)}
z^{n-2}+\ldots+a_0^{(n)}, \ n\in\N,$ be the Chebyshev polynomials
for $E$. Since $E$ is $\R$-symmetric, the coefficients of
Chebyshev polynomials are real (cf. \cite[p. 72]{Ts75}). By the
definition of (\ref{1.3}), for any $\eps>0$ there exists $N\in\N$
such that
$$ \|T_n\|_E^{1/n} \le \textup{cap}(E) + \eps, \quad n\ge N.$$
We shall construct a sequence of monic polynomials with integer
coefficients and small norms from the Chebyshev polynomials on
$E$. This is done by the following inductive procedure. Consider
$n\ge N$ and the polynomial $T_n-\left(a_{n-1}^{(n)}-
[a_{n-1}^{(n)}]\right)T_{n-1}$, with the two highest coefficients
integers. We have that
$$\|T_n-\left(a_{n-1}^{(n)}-[a_{n-1}^{(n)}]\right) T_{n-1}\|_E
\le (\textup{cap}(E)+\eps)^n + (\textup{cap}(E)+\eps)^{n-1}.$$
Continuing in the same fashion, we eliminate the fractional parts
of all coefficients from the $n$-th to $(N+1)$-st, and obtain the
following estimate
\begin{equation} \label{2.1}
\left\|z^n+[a_{n-1}^{(n)}]z^{n-1}+\ldots+[\cdot] z^{N+1} +
\sum_{k=0}^N b_k z^k \right\|_E \le \sum_{l=N+1}^n
(\textup{cap}(E)+\eps)^l.
\end{equation}
Note that
$$\left\|\sum_{k=0}^N (b_k-[b_k]) z^k \right\|_E \le \sum_{k=0}^N
\|z^k\|_E =: A(N),$$ where $A(N)>0$ depends only on $N$ and the
set $E$. Hence we have from (\ref{2.1}) that
\begin{align*}
\left\|z^n+[a_{n-1}^{(n)}]z^{n-1}+\ldots+[\cdot] z^{N+1} +
\sum_{k=0}^N [b_k] z^k \right\|_E & \\ \le
(\textup{cap}(E)+\eps)^{N+1} &
\frac{(\textup{cap}(E)+\eps)^{n-N}-1}{\textup{cap}(E)+\eps-1} +
A(N),
\end{align*}
because $\textup{cap}(E)\ge 1.$ Denote the constructed polynomial
by $Q_n\in\M(\Z),\ n\in\N.$ It follows that
$$\limsup_{n\to\infty} \|Q_n\|_E^{1/n} \le \textup{cap}(E)+\eps$$
and
$$t_M(E) \le \textup{cap}(E)+\eps.$$
Letting $\eps\to 0$ and recalling that $t_M(E) \ge
t_{\C}(E)=\textup{cap}(E)$ by definition, we finish the proof.
\end{proof}

\begin{proof}[Proof of Proposition \ref{prop1.2}]
See Kakeya's proof in \cite{Ok24} or \cite{Tr71}.
\end{proof}

\begin{proof}[Proof of Proposition \ref{prop1.3}]
This proposition readily follows from the inequality
$$\|p_n\|_E \le \|p_n\|_F,$$
valid for any polynomial $p_n(z).$
\end{proof}

\begin{proof}[Proof of Proposition \ref{prop1.4}]
The following argument is due to Fekete \cite{Fe30}. Let $M_k(z),\
k\in\N,$ be monic integer Chebyshev polynomials for $E$, and let
$M_k^*(z),\ k\in\N,$ be monic integer Chebyshev polynomials for
$E^*:=P_n^{-1}(E).$ It follows from the definition that
$$\|M_{kn}^*\|_{E^*} \le \|M_{k}\circ P_n\|_{E^*} \le
\|M_{k}\|_{E},\quad k\in\N.$$ Hence
$$ t_M(E^*) \le (t_M(E))^{1/n}.$$ To prove the opposite
inequality, we consider the roots $z_i,\ i=1,\ldots,n,$ of the
equation $P_n(z)-w=0$, where $w\in E$ is fixed. If $z_j^*,\
j=1,\ldots,k,$ are the roots of $M_k^*(z)$, then we have that
$$ \left|\prod_{i=1}^n M_k^*(z_i)\right| =
\left|\prod_{i=1}^n \prod_{j=1}^k (z_i-z_j^*)\right| =
\left|\prod_{j=1}^k \prod_{i=1}^n (z_j^*-z_i)\right| =
\left|\prod_{j=1}^k (P_n(z_j^*)-w)\right|. $$ Note that
$Q_k(w):=\prod_{j=1}^k (w-P_n(z_j^*))$ is a monic polynomial in
$w$, with integer coefficients. Indeed, its coefficients are
symmetric functions in $z_j^*$'s, which are integers by the
fundamental theorem on symmetric forms. Thus we obtain that
$$\|M_k\|_{E} \le \|Q_k\|_E \le (\|M_k^*\|_{E^*})^n,\quad k\in\N,$$
and
$$t_M(E) \le (t_M(E^*))^n.$$
\end{proof}

\begin{proof}[Proof of Theorem \ref{thm1.6}]
We first prove (\ref{1.12}). The sequence of polynomials
$\{z^k\}_{k=0}^{\infty}$ shows that $t_M([0,1/n])\le 1/n$ and
$t_M([-1/n,1/n])\le 1/n$. On the other hand, we have that
$t_M([0,1/n])\ge t_M(\{1/n\})=1/n$ by Proposition \ref{prop1.3}
and  (\ref{1.11}). The remaining equality for $t_M([1-1/n,1])$
follows from Proposition \ref{prop1.4}, by using the change of
variable $z\to 1-z.$

Applying the substitution $z\to z^2,$ we obtain (\ref{1.13}) from
Proposition \ref{prop1.4} and (\ref{1.12}).

Note that $t_M([0,1/2])=1/2$ follows from (\ref{1.12}). We can now
map $[0,1/2]$ to $[n,n+1/2]$ by $z\to z+n$ (or to $[n-1/2,n]$ by
$z\to n-z$), and apply Proposition \ref{prop1.4} to prove
(\ref{1.14}). Similarly, (\ref{1.15}) is obtained from Proposition
\ref{prop1.4} and (\ref{1.12}) by the transformations $z\to z-n$
mapping $[n,n+1] \to [0,1]$, and $z\to z(1-z)$ mapping $[0,1]\to
[0,1/4].$ The same argument applies to (\ref{1.16}), where we
first map $[n,n+2]\to [-1,1]$ with $z\to z-n-1$ and then map
$[-1,1]\to [0,1]$ with $z\to z^2.$

Observe that
$$|z(1-z)|\le 1/4,\quad z\in \left[\frac{1-\sqrt{2}}{2},
\frac{1+\sqrt{2}}{2}\right].$$ Hence
$$t_M\left(\left[\frac{1-\sqrt{2}}{2},\frac{1+\sqrt{2}}{2}\right]\right)
\le \frac{1}{2}.$$ For ${1/2} \in E \subset
[(1-\sqrt{2})/2,(1+\sqrt{2})/2]$, Proposition \ref{prop1.3} and
(\ref{1.11}) give that
$$\frac{1}{2}=t_M\left(\left\{\frac{1}{2}\right\}\right) \le
t_M\left(E\right) \le
t_M\left(\left[\frac{1-\sqrt{2}}{2},\frac{1+\sqrt{2}}{2}\right]\right)
\le \frac{1}{2}.$$ It is clear that the segment
$\left[\frac{1-\sqrt{2}}{2}, \frac{1+\sqrt{2}}{2}\right]$ can be
replaced here by the lemniscate $\{z\in\C:|z(1-z)|\le 1/4\}.$
\end{proof}

\section{Proofs for section 2}

\begin{proof}[Proof of Theorem \ref{thm2.1}]
Set
\begin{eqnarray*}
 E(j) & := & \prod_{i<j}  (a_jb_i-a_ib_j), \;\;\; j=2,\ldots,k, \\
D & := & lcm [E(2),\ldots,E(k)],
\end{eqnarray*}
and write $D:=D_1(j)D_2(j)$, $E(j):=E_1(j)E_2(j)$ where
\begin{align*} D_1(j) & =\prod_{\substack{ p^{\alpha}||D \\ p|b_j}}
p^{\alpha},\;\;\; &   D_2(j) & =\prod_{\substack{ p^{\alpha}||D \\
p\nmid b_j}} p^{\alpha}, \\
E_1(j) & =\prod_{\substack{ p^{\alpha}||E(j) \\ p|b_j}}
p^{\alpha},\;\;\;& E_2(j)& =\prod_{\substack{ p^{\alpha}||E(j) \\
p\nmid b_j}} p^{\alpha}.
\end{align*}
Take $m$ to be a positive integer large enough that
$$  p^{\alpha}|D \Rightarrow {\alpha}<m $$
and choose $n\geq km$ such that for $j=1,\ldots,k$
\begin{eqnarray} a_j^n & \equiv &  1 \mod b_j^{km},  \\
     b_j^n & \equiv &  1 \mod D_2(j)^{km}.  \end{eqnarray}
Choose integers $l_i$ such that
$$ a_il_i \equiv 1 \mod b_i,  $$
and write $a_il_i-b_if_i=1$.

The proof proceeds by induction on the number of rationals $1\leq
r\leq k$, constructing a polynomial
$$ F_r(x)=x^n+\sum_{i=0}^{n-(k+1-r)m} \beta_{i,r} x^i $$
with  $F_r(a_j/b_j)=1/b_j^n$, $j=1,\ldots,r$.

The first step, $r=1$, is easy;
$$ F_1(x):=x^n+\left(\frac{1-a_1^n}{b_1^{km}}\right)(l_1x-f_1)^{n-km}. $$

Next, given $F_r(x)$ with $r<k$ we construct $F_{r+1}(x)$. This
amounts to finding an integer polynomial $Q(x)$, of degree at most
$n-(k-r)m-r$ such that
$$ F_{r+1}(x)=F_r(x)+Q(x)\prod_{i=1}^r (b_ix-a_i) $$
has $F_{r+1}(a_{r+1}/b_{r+1})=1/b_{r+1}^n.$ Thus it is enough if
$$ \frac{1}{b_{r+1}^n}=F_r\left(\frac{a_{r+1}}{b_{r+1}}\right)+\frac{E(r+1)}
{b_{r+1}^r} \frac{A}{b_{r+1}^{n-(k-r)m-r}}, $$ for some integer
$A$  since we can then take
$$Q(x)=A\left(l_{r+1}x-f_{r+1}\right)^{n-(k-r)m-r}. $$
For this we require that $b_{r+1}^{(k-r)m } E(r+1)$ divides
$$ B:=b_{r+1}^{n}F_r\left(\frac{a_{r+1}}{b_{r+1}}\right)-1 =
(a_{r+1}^n-1)+\sum_{j=0}^{n-(k+1-r)m}
\beta_{j,r}a_{r+1}^jb_{r+1}^{n-j}. $$ Clearly from
$E_1(r+1)|D_1(r+1)$ and the definition of $m$ we have
$E_1(r+1)|b_{r+1}^m$, and
$b_{r+1}^{(k-r)m}E_1(r+1)|b_{r+1}^{(k+1-r)m}$. So from (4.1) and
$r\geq 1$  we certainly have that $b_{r+1}^{(k-r)m } E_1(r+1)$
divides $(a_{r+1}^n-1)$, and $b_{r+1}^{n-j}$ for $j\leq
n-(k+1-r)m$, and hence  $B$. Thus it remains to check that
$E_2(r+1)$ divides $B$.

Suppose that $p^{\alpha}||E_2(r+1)$. Then
$p|(b_ja_{r+1}-a_{j}b_{r+1})$ for some non-empty subset, $S$ say,
of the $1\leq j\leq r$. Note that since $p\nmid b_{r+1}$ we have
$p\nmid b_j$ and $ b_j^n \equiv 1 \mod p^{mk}$ for all $j\in S\cup
\{r+1\}$ from (4.2). Hence,  choosing integers $\overline{b}_{j}$
with $b_j\overline{b}_{j}\equiv 1 \mod p^{mk}$ for  $j\in
S\cup\{r+1\}$, we have
$$ 0 = b_j^{n}F_r(a_j/b_j)-1 \equiv F_r(a_j\overline{b}_{j})-1 \mod p^{mk}, $$
for the $j\in S$,  with $B\equiv F_r(a_{r+1}\overline{b}_{r+1})-1
\mod p^{mk}$ and $p^{\alpha}||\prod_{j\in
S}(a_{r+1}\overline{b}_{r+1}-a_{j}\overline{b}_{j})$.
 Thus we can  successively divide
$F_r(x)-1$ by $(x-a_j\overline{b}_j)$ for the $j\in S$ (assume we
proceeed in order of increasing $j$). In particular  after dealing
with a subset $S'$ of the $j$ in $S$ we can write
$$ F_r(x)-1 \equiv T_{S'}(x) \prod_{j\in S'} (x-a_j\overline{b}_j) \mod p^{(k+1-|S'|)m}$$
for some integer polynomial $T_{S'}(x)$, where $p^m \nmid
\prod_{j\in S'}(a_{i}\overline{b}_{i}-a_{j}\overline{b}_{j})$ (as
$p^m \nmid E(i)$) and $F_r(a_i\overline{b}_{i}) -1 \equiv 0 \mod
p^{km}$ imply that
$$T_{S'}(a_i\overline{b}_i)\equiv 0 \mod p^{(k+1-|S'|-1)m}$$
for any remaining $i\in S\setminus S'$. So
\begin{align*} B & \equiv T_S(a_{r+1}\overline{b}_{r+1}) \prod_{j\in S}(a_{r+1}\overline{b}_{r+1}-a_{j}\overline{b}_{j}) \mod p^{(k+1-|S|)m} \\
  &  \equiv 0 \mod p^{\alpha},
\end{align*}
as claimed.
\end{proof}

\begin{proof}[Proof of Theorem \ref{thm2.2}]
Suppose that we have a set of $k$ numbers as in the statement of
Theorem \ref{thm2.2}.

We first show that for any $1>\ve >0$ there is a non-zero integer
polynomial, $P(x)=x^j Q(x)$ with $j\le\binom{k}{2}$ and $Q$ of
degree at most $k$, with $0<|P(\alpha_i)|<\ve/k$, $i=1,\ldots,k,$
and $P(\alpha_i)\neq P(\alpha_l)$ when $\alpha_i/\alpha_l$ is not
a root of unity. This essentially follows from Minkowski's theorem
on linear forms: Taking an arbitrary real $\alpha_{k+1}\neq
\alpha_i$, $i=1,\ldots,k$ we can find a non zero
$(a_0,\ldots,a_k)\in \Z^{k+1}$ with
$$ |a_0+a_1\alpha_i+\cdots + a_k\alpha_i^{k}| \leq
\frac{\ve}{k\max\{1,|\alpha_i|\}^{\frac{1}{2}k(k-1)}}, $$ if
$\alpha_i$, $i=1,\ldots,k$, is real, and for any pairs of complex
conjugate $\alpha_i$
$$ |a_0+a_1\Re \alpha_i+\cdots + a_k\Re\alpha_i^{k}|
\leq \frac{\ve}{\sqrt{2}k\max\{1,|\alpha_i|\}^{\frac{1}{2}k(k-1)}}, $$
$$ |a_0+a_1\Im \alpha_i+\cdots + a_k\Im \alpha_i^{k}| \leq
\frac{\ve}{\sqrt{2}k\max\{1,|\alpha_i|\}^{\frac{1}{2}k(k-1)}}, $$
and
$$ |a_0+a_1\alpha_{k+1}+\cdots + a_k\alpha_{k+1}^{k}| \leq
\frac{Dk^k\prod_{i=1}^k\max\{1,|\alpha_i|\}^{\frac{1}{2}k(k-1)}
}{\ve^k}, $$ where
$$ D=\left|\det \begin{pmatrix} 1 & \cdots & \alpha_1^k \\
\vdots & & \vdots \\ 1 & \cdots & \alpha_{k+1}^k
 \end{pmatrix} \right|=\prod_{i<j} |\alpha_i-\alpha_j| \neq 0. $$
Taking $Q(x)=a_0+\cdots +a_kx^k$ we plainly have $Q(\alpha_i)\neq
0$ (since $[\Q(\alpha_i):\Q]>k$) and $\alpha_i^jQ(\alpha_i)\neq
\alpha_l^jQ(\alpha_l)$, when $\alpha_i/\alpha_l$ is not a root of
unity, for at least one $0\leq j\leq k(k-1)/2$ (since when
$\alpha_i/\alpha_l$ is not a root of unity
$Q(\alpha_l)/Q(\alpha_i)=(\alpha_i/\alpha_l)^j$ for at most one
integer $j$, and there are at most $k(k-1)/2$ such pairings
$i<l$). Choosing $P(x)=x^jQ(x)$ for such a  $j$ then has the
desired property.

To complete the proof of Theorem \ref{thm2.2} take the polynomial
$P(x)$ as above, and an $n>k^2(k+1)/2$ such that
$\alpha_i^n=\alpha_l^n$ whenever $\alpha_i/\alpha_l$ is a root of
unity, and solve the linear system
$$ A_1P(\alpha_i)+A_2 P(\alpha_i)^2+ \cdots + A_mP(\alpha_i)^m =
-\alpha_i^n,\;\;\; i=1,\ldots,m,$$ where
$P(\alpha_1),\ldots,P(\alpha_m)$ are the distinct values of
$P(\alpha_i)$ (any remaining $\alpha_l$ with
$P(\alpha_l)=P(\alpha_i),\ \alpha_l^n=\alpha_i^n$ will merely
repeat one of these equations). This will have a solution, since
$$ \left|\det \begin{pmatrix} P(\alpha_1) & \cdots & P(\alpha_1)^{m} \\
\vdots  & & \vdots  \\ P(\alpha_m) & \cdots & P(\alpha_m)^{m}
\end{pmatrix} \right| =|P(\alpha_1)|\cdots |P(\alpha_m)|\prod_{i<j}
|P(\alpha_i)-P(\alpha_j)|\neq 0. $$ Moreover, since the complex
$\alpha_i$ come in complex conjugate pairs, the solution
$A_1$,\ldots,$A_m$ will be real. Hence taking $b_j=[A_j]$,
$j=1,\ldots,m,$ gives a monic integer polynomial
$$ F(x)=x^n+b_mP(x)^m+b_{m-1}P(x)^{m-1}+\cdots +b_1P(x) $$
with
$$ |F(\alpha_i)|=\left|\sum_{j=1}^m \{A_j\}P(\alpha_i)^j\right|\leq
\sum_{j=1}^m (\ve/k)^j <\ve. $$

\end{proof}

\section{Intervals of Consecutive Farey Numbers}

\begin{con} \label{con}
Suppose $[{a_2}/{b_2},{a_1}/{b_1}]$ is an interval whose endpoints
are consecutive Farey fractions. This is characterized by
$a_1b_2-a_2b_1=1.$ Then  $$t_M[{a_2}/{b_2},{a_1}/{b_1}] =
\max(1/b_1,1/b_2).$$
\end{con}

From Corollary \ref{cor2.1}  $$t_M[{a_2}/{b_2},{a_1}/{b_1}] \ge
\max(1/b_1,1/b_2)$$ and from  Theorem \ref{thm1.6} the conjecture
holds on intervals of the form $[0,1/n]$. The following table
gives enough solutions to fill in all Farey intervals with
denominator less than 22. On the remaining intervals $x$ works or
the symmetry $x\rightarrow m\pm x$ works.

The computations for the table are done with LLL. As in section 2,
for certain $n$, we can find a polynomial $p$ of degree $n$ that
satisfies $p({a_2}/{b_2}) =1/b_2^n$ and $p({a_1}/{b_1}) =1/b_1^n.$
One now constructs a basis
$$B:=\left(p(x), (b_1 x- a_1)(b_2 x -a_2), x (b_1 x- a_1)(b_2 x -a_2),
\ldots, x^{n-3} (b_1 x- a_1)(b_2 x -a_2)\right)$$ and we reduce
the basis with respect to the norm
$$\left(\int_{{a_2}/{b_2}}^{{a_1}/{b_1}} p(x)^2
\,dx\right)^{1/2}.$$ We then search the reduced basis for
solutions of the conjecture. These calculations were done in Maple
using an LLL implementation that can accommodate  reduction with
respect to any positive definite quadratic form. (This was
implemented by Kevin Hare and we would like to thank him for his
code.)

Here $T({a_2}/{b_2},{a_1}/{b_1})$ is a polynomial that satisfies
\[ \norm{T({a_2}/{b_2},{a_1}/{b_1})}_{[{a_2}/{b_2},{a_1}/{b_1}]} =
\max(1/b_1,1/b_2)^{\deg T}, \] so that Conjecture \ref{con} holds
on $[{a_2}/{b_2},{a_1}/{b_1}]$ by Lemma \ref{new}. There is no
guarantee that it is the lowest degree example.

\begin{equation}
\begin{array}{lcl} \notag
T(1/3,2/5) & = & {x}^{2}-3\,x+1 \\
T(1/4,2/7) & = & {x}^{2}-4\,x+1 \\
T(2/5,3/7) & = & {x}^{4}-716\,{x}^{3}+890\,{x}^{2}-369\,x+51 \\
T(1/3,3/8) & = & {x}^{2}-3\,x+1 \\
T(3/8,2/5) & = & {x}^{2}-3\,x+1 \\
T(1/5,2/9) & = & -{x}^{3}-20\,{x}^{2}+9\,x-1 \\
T(3/7,4/9) & = & -{x}^{3}+37\,{x}^{2}-32\,x+7 \\
T(2/7,3/10) & = & -{x}^{6}+1151931\,{x}^{5}-1691236\,{x}^{4}+993150\,{x}^{3}-291587\,{x}^{2}\\
& &+42802\,x-2513 \\
T(1/6,2/11) & = & -{x}^{3}-30\,{x}^{2}+11\,x-1 \\
T(1/4,3/11) & = & {x}^{2}-4\,x+1 \\
T(3/11,2/7) & = & -{x}^{6}-2359829\,{x}^{5}+3291253\,{x}^{4}-1836029\,{x}^{3}+512089\,{x}^{2}\\
&&-71410\,x+3983 \\
T(1/3,4/11) & = & {x}^{2}-6\,x+2 \\
T(4/11,3/8) & = & -{x}^{4}+830\,{x}^{3}-928\,{x}^{2}+346\,x-43 \\
T(4/9,{ {5 / 11}}) & = &
-{x}^{9}-29635158678\,{x}^{8}+106792009997\,{x}^{7}-
168361710540\,{x}^{6} \\
 & & +151671807240\,{x}^{5}-85396766648\,{x}^{4}+30771806151\,{x}^{3} \\
 &&-6930101424\,{x}^{2}+891832252\,x-50211113 \\
T(2/5,{ {5 / 12}}) & = & {x}^{2}+2\,x-1 \\
T({ {5 / 12}},3/7) & = & -{x}^{3}-26\,{x}^{2}+23\,x-5 \\
T(1/7,2/13) & = & -{x}^{3}-42\,{x}^{2}+13\,x-1 \\
T(2/9,3/13) & = & -{x}^{3}-20\,{x}^{2}+9\,x-1 \\
T(3/10,{ {4 / 13}}) & = &
-{x}^{6}-2627119\,{x}^{5}+3994979\,{x}^{4}-2429980\,{x}^{3}+739017\,{x}^{2}\\
&&-112375\,x+6835 \\
T(3/8,{ {5 / 13}}) & = & {x}^{2}-3\,x+1 \\
T({ {5 / 13}},2/5) & = & {x}^{2}-3\,x+1 \\
T({ {5 / 11}},{ {6 / 13}}) & = &
{x}^{5}+131482\,{x}^{4}-240886\,{x}^{3}+165494\,{
x}^{2}-50532\,x+5786 \\
T(1/5,3/14) & = & {x}^{2}-5\,x+1 \\
T(3/14,2/9) & = & -{x}^{3}+106\,{x}^{2}-46\,x+5 \\
T(1/3,{ {5 / 14}}) & = & {x}^{2}-6\,x+2 \\
T({ {5 / 14}},4/11) & = &
{x}^{5}+98683\,{x}^{4}-142309\,{x}^{3}+76957\,{x}^{2}-18496
\,x+1667 \\
T(1/8,2/15) & = & -{x}^{4}+4112\,{x}^{3}-1602\,{x}^{2}+208\,x-9 \\
T(1/4,{ {4 / 15}}) & = & {x}^{2}-8\,x+2 \\
T({ {4 / 15}},3/11) & = &
{x}^{5}+162382\,{x}^{4}-175226\,{x}^{3}+70906\,{x}^{2}-12752
\,x+860\\
T({ {6 / 13}},{ {7 / 15}}) & = &
-{x}^{12}+72422527702901325\,{x}^{11}-
369852358365610457\,{x}^{10}\\ &&+858538529519890462\,{x}^{9}
-1195753892838600326\,{x}^{8}\\ && +1110278480747979603\,{x}^{7}
-721638086761400063\,{x}^{6}
\\ & &+335025835998692775\,{x}^{5}-111098573305754871\,{x}^{4}\\
&&+25789093045603361\,{x}^{3}-3990908419523891\,{x}^{2} \\ &&+
370559601060925\,x-15639435102355\\
T(2/11,3/16) & = & {x}^{5}-16175\,{x}^{4}+12295\,{x}^{3}-3502\,{x}^{2}+443\,x-21 \\
T({ {4 / 13}},{ {5 / 16}}) & = &
{x}^{9}-369076253174\,{x}^{8}+917724840702\,{x}^{7}-998350365312\,{x}^{6}\\
&&+620599596183\,{x}^{5}-241110478731\,{x}^{4}+59951042224\,{x}^{3}\\
&&-9316515227\,{x}^{2}+827310070\,x-32140845 \\
T(3/7,{ {7 / 16}}) & = & {x}^{3}-37\,{x}^{2}+32\,x-7 \\
T({ {7 / 16}},4/9) & = &
-{x}^{6}-10576186\,{x}^{5}+23312009\,{x}^{4}-20553597\,{x}^{3}\\
&&+9060737\,{x}^{2}-1997132\,x+176079 \\
T(1/9,2/17) & = & {x}^{4}-1953\,{x}^{3}+676\,{x}^{2}-78\,x+3
\end{array}
\end{equation}

\begin{equation}
\begin{array}{lcl} \notag
T(1/6,{ {3 / 17}}) & = & {x}^{2}-6\,x+1\\
T({ {3 / 17}},2/11) & =
& {x}^{5}-164752\,{x}^{4}+117596\,{x}^{3}-31475\,{x}^{2}+3744
\,x-167\\
T(3/13,{ {4 / 17}}) & = &
-{x}^{6}+5654596\,{x}^{5}-6591735\,{x}^{4}+3073650\,{x}^{3}-
716598\,{x}^{2}\\ &&+83534\,x-3895 \\
T(2/7,{ {5 / 17}}) & = & -{x}^{3}-82\,{x}^{2}+48\,x-7 \\
T({ {5 / 17}},3/10) & = & -{x}^{4}-3986\,{x}^{3}+3547\,{x}^{2}-1052\,x+104 \\
T(1/3,{ {6 / 17}}) & = & {x}^{2}-6\,x+2 \\
T({ {6 / 17}},{ {5 / 14}}) & = &
-{x}^{12}-312068777674512248\,{x}^{11}+
1218635449662069926\,{x}^{10}\\
&&-2163086830555697775\,{x}^{9}+2303693455082796817\,{x}^{8}\\
&&-1635624110167518281\,{x}^{7}+812904818934872080\,{x}^{6}\\
&&-288580642413016261\,{x}^{5}+73175550293900447\,{x}^{4}\\
&&-12988595068142207\,{x}^{3}+1536973706418270\,{x}^{2}\\
&&-109124254036404\,x+3521703559324 \\
T(2/5,{ {7 / 17}}) & = & {x}^{2}+2\,x-1 \\
T({ {7 / 17}},{ {5 / 12}}) & = & {x}^{2}+2\,x-1 \\
T({ {7 / 15}},{ {8 / 17}}) & = & -{x}^{4}+6130\,{x}^{3}-8618\,{x}^{2}+4039\,x-631 \\
T(3/11,{ {5 / 18}}) & = &
{x}^{5}+303655\,{x}^{4}-334282\,{x}^{3}+137998\,{x}^{2}-
25319\,x+1742 \\
T({ {5 / 18}},2/7) & = & -{x}^{3}-89\,{x}^{2}+50\,x-7 \\
T({ {5 / 13}},{ {7 / 18}}) & = &
-{x}^{6}-6049372\,{x}^{5}+11706532\,{x}^{4}-
9061607\,{x}^{3}+3507125\,{x}^{2}\\
&&-678682\,x+52534 \\
T({ {7 / 18}},2/5) & = & {x}^{2}-3\,x+1 \\
T(1/10,2/19) & = & {x}^{4}-2710\,{x}^{3}+841\,{x}^{2}-87\,x+3 \\
T(2/13,{ {3 / 19}}) & = &
{x}^{6}+16300632\,{x}^{5}-12702977\,{x}^{4}+3959686\,{x}^{3}
-617135\,{x}^{2}\\ &&+48091\,x-1499 \\
T(1/5,{ {4 / 19}}) & = & {x}^{2}-5\,x+1 \\
T({ {4 / 19}},3/14) & = &
{x}^{9}-2941000101126\,{x}^{8}+4994011925448\,{x}^{7}
\\ & &
-3710051448922\,{x}^{6}
+1574961728536\,{x}^{5}-417866792428\,{x}^{4}
\\ & &
+70955073227\,{x}
^{3}-7530205493\,{x}^{2}+456656790\,x-12115709 \\
T(1/4,{ {5 / 19}}) & = & {x}^{2}-4\,x+1 \\
T({ {5 / 19}},{ {4 / 15}}) & = & {x}^{4}+3607\,{x}^{3}-2866\,{x}^{2}+759\,x-67 \\
T({ {5 / 16}},{ {6 / 19}}) & = &
-{x}^{8}+51931371494\,{x}^{7}-114207671161\,{x}^{
6}+107642363378\,{x}^{5}
\\ & &
-56363495447\,{x}^{4}+17707739051\,{x}^{3}-3337942176\,{x}^{ 2}
\\ & &
+349559613\,x-15688671 \\
T(4/11,{ {7 / 19}}) & = &
-{x}^{5}-167972\,{x}^{4}+246584\,{x}^{3}-135743\,{x}^{2}+
33211\,x-3047\\
T({ {7 / 19}},3/8) & = &
{x}^{4}+8594\,{x}^{3}-9574\,{x}^{2}+3555\,x-440\\
T({ {5 / 12}},{{8 / 19}}) & = &
{x}^{6}-13761534\,{x}^{5}+28842552\,{x}^{4}
\\ & &
- 24180158\,{x}^{3}+10135687\,{x}^{2}-2124300\,x+178089 \\
T({ {8 / 19}},3/7) & = & {x}^{3}-107\,{x}^{2}+90\,x-19\\
T({ {8 / 17}},{ {9 / 19}}) & = &
-{x}^{8}+63880292236\,{x}^{7}-211132804023\,{x}^{
6}+299066280893\,{x}^{5}
\\ & &
-235345759625\,{x}^{4}+111121028980\,{x}^{3}-31480170773\,{x }^{2}
\\ & &
+4954560070\,x-334191985 \\
T(1/7,{ {3 / 20}}) & = & {x}^{3}+231\,{x}^{2}-68\,x+5 \\
T({ {3 / 20}},2/13) & = &
-{x}^{6}+28336792\,{x}^{5}-21517541\,{x}^{4}+6535640\,{x}^{3
}-992538\,{x}^{2}\\ &&+75365\,x-2289 \\
T(1/3,{ {7 / 20}}) & = & {x}^{2}-6\,x+2 \\
T({ {7 / 20}},{ {6 / 17}}) & = &
{x}^{8}-70615270260\,{x}^{7}+173702478683\,{x}^{6
}-183120134018\,{x}^{5}
\\ & &
+107248975216\,{x}^{4}-37687812630\,{x}^{3}+7946199062\,{x}^{2}
\\ & &
-930775452\,x+46725407
\end{array}
\end{equation}

\begin{equation}
\begin{array}{lcl} \notag
T(4/9,{ {9 / 20}}) & = & {x}^{3}+224\,{x}^{2}-201\,x+45 \\
T({ {9 / 20}},{ {5 / 11}}) & = &
-{x}^{5}-28247\,{x}^{4}+50444\,{x}^{3}-33775\,{x}
^{2}+10049\,x-1121\\
T(1/11,2/21) & = & {x}^{4}-3641\,{x}^{3}+1024\,{x}^{2}-96\,x+3 \\
T(3/16,{ {4 / 21}}) & = &
-{x}^{8}+136425013870\,{x}^{7}-180508372914\,{x}^{6}+
102358062346\,{x}^{5}
\\ & &
-32245743882\,{x}^{4}+6094977959\,{x}^{3}-691227923\,{x}^{2}\\
&&+43550791\,x-1175959 \\
T({ {4 / 17}},{ {5 / 21}}) & = &
{x}^{6}+1414956\,{x}^{5}-1686559\,{x}^{4}+804089
\,{x}^{3}-191673\,{x}^{2}\\
&&+22844\,x-1089 \\
T(3/8,{ {8 / 21}}) & = & {x}^{2}-3\,x+1 \\
T({ {8 / 21}},{ {5 / 13}}) & = & {x}^{2}-3\,x+1 \\
T({ {9 / 19}},{ {10 / 21}}) & = &
{x}^{18}-265066219851470073896475021927\,{x}^{17}\\
&&+2140395330694655830972341091874\,{x}^{16}\\
&&-8133445821830247750162364615479\,{x}^{15}\\
&&+19316795672890032633988244072508\,{x}^{14}\\
&&-32113936273710937029720760450948\,{x}^{13}\\
&&+39660410718965991151182638887921\,{x}^{12}\\
&&-37677096594660667022412296504028\,{x}^{11}\\
&&+28123036244133465310172098724688\,{x}^{10}\\
&&-16697915529194766473201489538076\,{x}^{9}\\
&&+7931442994928916901189904965470\,{x}^{8}\\
&&-3013922280150590577654465661841\,{x}^{7}\\
&&+911018436460175951387551399941\,{x}^{6}\\
&&-216364915651909887212093381346\,{x}^{5}\\
&&+39527838685420394912701179067\,{x}^{4}\\
&&-5364441433555090728913121916\,{x}^{3}\\
&&+509616986326961914742507595\,{x}^{2}\\
&&-30258208210601324759757834\,x\\
&&+845441362748491768882081
\end{array}
\end{equation}

\end{document}